\documentclass[10pt, a4paper, reqno, english]{amsart}
\pdfoutput=1
\usepackage{amsfonts, amsthm, amsmath, amssymb, mathrsfs}

\usepackage[utf8]{inputenc}
\usepackage{lmodern}
\usepackage{array}
\usepackage[margin=3.5cm]{geometry}
\usepackage{enumitem}
\usepackage[draft=false,hidelinks]{hyperref}
\usepackage{graphics}
\usepackage{dsfont}
\usepackage{tikz}
\usetikzlibrary{shapes.geometric}

\usepackage{booktabs}
\usepackage{caption}
\usepackage[all]{xy}

\theoremstyle{plain}
\newtheorem{theorem}{Theorem}[section]

\newtheorem{lemma}[theorem]{Lemma}
\newtheorem{corollary}[theorem]{Corollary}

\theoremstyle{definition}

\newtheorem*{xremark}{Remark}
\newtheorem*{xnotation}{Notation}

\numberwithin{equation}{section}

\makeatletter
\@namedef{subjclassname@2020}{\textup{2020} Mathematics Subject Classification}
\makeatother

\begin{document}

\title[A density version for the sum of nine cubes of primes]{A density version for the sum of nine cubes of primes}

\author{Kaijie Li}
\address{School of Mathematics and Statistics, North China University of Water Resources and Electric Power, Zhengzhou, Henan 450046, P.R. China}
\email{lkjshw@163.com}
\author{Tianze Wang}
\address{Institute of Mathematics, Henan Academy of Sciences, Zhengzhou 450046, Henan, P. R. China\\
School of Mathematics and Statistics, North China University of Water Resources and Electric Power, Zhengzhou 450046, Henan, P. R. China}
\email{wtz@ncwu.edu.cn}
\author{Xiaodong Zhao*}
\address{Institute of Mathematics, Henan Academy of Sciences, Zhengzhou 450046, Henan, P. R. China}
\email{zhaoxiaodong@hnas.ac.cn}


\date{}

\thanks{*corresponding author}
\setcounter{tocdepth}{1}

\begin{abstract}
Let $\mathbb{P}$ be the set of primes. Suppose that $P_{i}$, $i=1,2,\ldots,9$ are nine subsets of $\mathbb{P}$ with relative lower density $\underline{d}_{\mathbb{P}}(P_{i})=\mu_{i}$. When $\sum_{i=1}^{9} \mu_{i}^{3}>8$, we prove that every sufficiently large odd integer can be written as $n=\sum_{i=1}^{9} p_{i}^{3}$ with $p_{i} \in P_{i}$, $i=1,2,\ldots,9$.
\end{abstract}

\maketitle
\tableofcontents

\section{Introduction}
In 1930s, mathematicians achieved remarkable results on the Waring–Goldbach problem. The celebrated Goldbach conjecture states that every even integer $n\geq 4$ can be written as the sum of two primes,
which is the most representative in the linear Waring–Goldbach framework. This problem is generally considered to lie beyond the scope of current methods. In contrast, the ternary Goldbach conjecture—which asserts that every sufficiently large odd integer is the sum of three primes—is significantly more accessible. In 1937, Vinogradov \cite{vinogradov1937} proved that every sufficiently large odd integer can be represented as the sum of three primes. Subsequently, efforts were directed toward proving a density version of Vinogradov’s three primes theorem. More precisely, let $\mathbb{P}$ denote the set of all primes. Suppose $A\subset \mathbb{P}$, define the relative lower density by
$$\underline{d}_{\mathbb{P}}(A)=\liminf\limits_{x \to \infty} \frac{|A \cap[1, x]|}{|\mathbb{P} \cap[1, x]|}.$$
In 2010, suppose that $P_{1}, P_{2}, P_{3}$ are three subsets of $\mathbb{P}$ with
$\underline {{d}}_{\mathbb{P}}(P_{1})+\underline {{d}}_{\mathbb {P}}(P_{2})+\underline {{d}}_{\mathbb {P}}(P_{3})>2,$ Li and Pan \cite{lipan2010} showed that every sufficiently large odd integer $n$ can be written as the sum $n=p_{1}+p_{2}+p_{3}$ with $p_{1} \in P_{1}, p_{2} \in P_{2}, p_{3} \in P_{3}$. In 2014, let $P_{0} \subset \mathbb{P}$ be a subset of the primes with $\underline{d}_{\mathbb{P}}(P_{0})>\frac{5}{8}$, Shao \cite{shao2014} proved that for any sufficiently large odd positive integer $N$ there exists $p_{1}, p_{2}, p_{3} \in P_{0}$ with $N=p_{1}+p_{2}+p_{3}$.

For the quadratic Waring–Goldbach problem, in 1938, Hua \cite{Hua} proved that any sufficiently large number $n \equiv 5 \pmod {24}$ can be written as the sum of five prime squares.
In 2025, for $i=1,2,3,4,5$, suppose $P_{i}$ are five subsets of the prime set $\mathbb{P}$ with lower relative density $\underline{d}_{\mathbb{P}}(P_{i})=\varepsilon_{i}$, and $\sum_{i=1}^{5} \varepsilon_{i}^{2}>4$, Tan \cite{tan} showed that for any sufficiently large integer $n \equiv 5 \pmod {24}$, there exist $p_{i} \in P_{i}$ such that $n=\sum_{i=1}^{5} p_{i}^{2}$.\\

For the general Waring–Goldbach problem, let $k\in\mathbb{Z}_{\geq 2}$ and $\underline{d}_{\mathbb{P}}(A) > 1 - 1/2k$, in 2025, Gao \cite{Gao} proved that every sufficiently large natural number $n$ satisfying the necessary congruence condition can be written in the form
$$n = p_1^k + \cdots + p_s^k,$$
where $s > \max\left(16k\omega(k) + 4k + 3, C(k)\right)$ and $p_i \in A$ for all $i \in \{1, \ldots, s\}$. Here, $\omega(k)$ denotes the number of distinct prime divisors of
$k$ and
$$
C(k) =
\begin{cases}
\lfloor k^2 / 2 \rfloor & \text{if } k \neq 4, \\
7 & \text{if } k = 4.
\end{cases}
$$

For the cubic Waring–Goldbach problem, in 1965, Hua \cite{Hua1} proved that every sufficiently large odd integer can be expressed as the sum of nine cubes of primes. Reasonably, in this paper,
from the idea of the transference principle of Green \cite{green} and the strategy of Tan \cite{tan}, we will prove a density version of of Hua’s theorem.
\begin{theorem}\label{THM2}
For $i = 1, 2, \ldots ,9$, suppose $P_i$ are nine subsets of the prime set $\mathbb{P}$, with lower relative density $\underline{d}_{\mathbb{P}}(P_i) = \mu_i$, and suppose that $\sum_{i=1}^9 \mu_i^3 > 8$. Then, for any sufficiently large odd integer, there exist
$p_i \in P_i$ such that $n = \sum_{i=1}^9p_i^3$.
\end{theorem}
\begin{xremark}
The result of Theorem \ref{THM2} is the best possible. Let $P_i = \mathbb{P} \setminus \{2\}$ for $i = 1, 2, \ldots, 8$, and $P_9 = \{2\}$. Then, it follows that $\mu_1 = \mu_2 = \cdots = \mu_8 = 1$, $\mu_9 = 0$, and $\sum_{i=1}^9 \mu_i^3 = 8$. For $p_i \in P_i$, $i = 1, 2, \ldots, 9$, we obtain $n = \sum_{i=1}^8 p_i^3 + 8$, which contradicts that $n$ is an odd number.
\end{xremark}
\begin{theorem}\label{THM3}
For $i = 1, 2, \ldots ,s$ and $s\geq 10$, suppose $P_i$ are $s$ subsets of the prime set $\mathbb{P}$, with lower relative density $\underline{d}_{\mathbb{P}}(P_i) = \mu'_i$, and suppose that $\sum_{i=1}^s {\mu'_i}^3 > s-1$. Then, for any sufficiently large odd integer, there exist
$p_i \in P_i$ such that $n = \sum_{i=1}^sp_i^3$.
\end{theorem}
\begin{xremark}
The proof of this theorem is similar to Theorem 1.1, we will omit it.
\end{xremark}
\begin{theorem}\label{THM1}
Let $q$ be a square-free integer and $(q, 14) = 1$. Let $f_1, f_2, \ldots, f_9$, be nine real-valued functions on $\mathbb{Z}_q^{(3)}$. Then, for any $n \in \mathbb{Z}_q$, there exist $x_1, x_2, \ldots, x_9 \in \mathbb{Z}_q^{(3)}$ such that $n \equiv x_1 + x_2 + \cdots + x_9\pmod q$ in $\mathbb{Z}_q$ and that
$$\sum_{i=1}^9f_i(x_i)\geq\frac{1}{\phi_3(q)}\sum_{a\in \mathbb{Z}_q^{(3)}}\sum_{i=1}^9f_i(a),$$
where $\mathbb{Z}_q^{(3)}$ and $\phi_3(q)$ are defined in the following Notation.
\end{theorem}
\begin{xnotation}
Given a positive integer $q$, we write
$\mathbb{Z}_q=\mathbb{Z}/q\mathbb{Z} \ \text{and}\ \mathbb{Z}_q^{\ast}=\{b\in\mathbb{Z}_q:(b,q)=1\}.$
We write $\phi(q)=|\mathbb{Z}_q^\ast|$ to be the Euler totient function. Moreover, let $\mathbb{Z}_q^{(3)} $ denote the set of cubic residues in $\mathbb{Z}_q^\ast$ and $\phi_3(q)=|\mathbb{Z}_q^{(3)}|$. Let $e(x)=e^{2 \pi i x}$. For any $n \neq0 \in \mathbb{Z}_{p}$, let $h(n)$ be the number of solutions $(x, y)$ to the equation $x^{3}+y^{3} \equiv n\pmod p$, where $0\neq x, y\in \mathbb{Z}_p$. Similarly, for any $n \neq0 \in \mathbb{Z}_{p}$, let $h'(n)$ be the number of solutions $(x, y, z)$ to the equation $x^{3}+y^{3}+z^{3} \equiv n\pmod p$, where $0\neq x, y, z\in \mathbb{Z}_p$. The letter $p$, with or without subscript, is reserved for a prime number. Let $\epsilon$ be an arbitrarily small positive number and the value of $\epsilon$ may change from line to line. Let $S_k$ denote the symmetric group on $k$ elements.
\end{xnotation}

\section{Some auxiliary results}
In this section, we will give some results, which will play a key role in the proof of Theorem \ref{THM1}. Now, we will give some properties on $h(n)$ and $h'(n)$ to prove
Lemmas \ref{lemma1}-\ref{lemma2} and Corollaries \ref{co1}-\ref{co2}. Fix a positive integer $n$ such that $(n, p) = 1$, we have
\begin{align*} h(n) & =\sum_{\substack{0 < x, y \leq p-1 \\ x^{3}+y^{3}\equiv n\pmod p}} 1 \\ & =\frac{1}{p} \sum_{0 < x, y \leq p-1} \sum_{1 \leq r \leq p} e\left(\frac{r(x^{3}+y^{3}-n)}{p}\right) \\ & =\frac{1}{p} \sum_{1 \leq r \leq p} e\left(\frac{-r n}{p}\right) \sum_{0 < x, y \leq p-1} e\left(\frac{r x^{3}}{p}\right) e\left(\frac{r y^{3}}{p}\right) \\ & =\frac{1}{p} \sum_{1 \leq r \leq p} e\left(\frac{-r n}{p}\right)\left(\sum_{0 < x \leq p-1} e\left(\frac{r x^{3}}{p}\right)\right)^{2}. \end{align*}
Now we divide the summation over $r$ into three cases: $\left(\frac{r}{p}\right)_3=0$, $\left(\frac{r}{p}\right)_3=1$ and $\left(\frac{r}{p}\right)_3=-1$. Here and in the sequel,
$$\left(\frac{r}{p}\right)_3 = \begin{cases}
     1, &r\, \text{is the cubic residue modulo}\, p,\\
      -1, &r\, \text{is not the cubic residue modulo}\,p,\\
      0,& r=p.
    \end{cases}$$
Then, we have
\begin{align*}
 h(n)=&\frac{1}{p}\Biggl((p-1)^{2}+\sum_{\left(\frac{r}{p}\right)_3=1} e\left(\frac{-r n}{p}\right)\biggl(\sum_{0 < x \leq p-1} e\left(\frac{r x^{3}}{p}\right)\biggr)^{2}\\
 &+\sum_{\left(\frac{r}{p}\right)_3=-1} e\left(\frac{-r n}{p}\right)\biggl(\sum_{0 < x \leq p-1} e\left(\frac{r x^{3}}{p}\right)\biggr)^{2}\Biggr) \\ =&\frac{1}{p}\Biggl((p-1)^{2}+\sum_{\left(\frac{r}{p}\right)_3=1} e\left(\frac{-r n}{p}\right)\biggl(\sum_{0 < x \leq p-1} e\left(\frac{x^{3}}{p}\right)\biggr)^{2}\\
 &+\sum_{\left(\frac{r}{p}\right)_3=-1} e\left(\frac{-r n}{p}\right)\biggl(\sum_{0 < x \leq p-1} e\left(\frac{r' x^{3}}{p}\right)\biggr)^{2}\Biggr),
 \end{align*}
where $r'$ is a fixed non-cubic residue modulo $p$.

Since all primes greater than 3 have only two congruence classes modulo 3, i.e. $p\equiv1\pmod 3$ and $p\equiv2\pmod 3$, we will prove the following results by dividing into these
two cases respectively.

For $p\geq13$, $p\equiv 1\pmod 3$, by the inequality
$$\left|\sum_{x=0}^{p-1} e\left(\frac{r x^{3}}{p}\right)\right|\leq 2p^{\frac{1}{2}}$$
for $1 \leq r \leq p-1$, we have
$$\left|\sum_{\left(\frac{r}{p}\right)_3=1} e\left(\frac{r}{p}\right)\right| ,\quad\left|\sum_{\left(\frac{r}{p}\right)_3=-1} e\left(\frac{r}{p}\right)\right| \leq \frac{2p^\frac{1}{2}+2}{3}.$$
Finally, we have
\begin{align}\label{hn}h(n) \geq \frac{1}{p}\left((p-1)^{2}-\frac{4}{3}\left(2p^{\frac{1}{2}}+1\right)^2\left(p^{\frac{1}{2}}+1\right)\right).\end{align}
Similarly, we have
\begin{align}\label{hnn}h'(n) \geq \frac{1}{p}\left((p-1)^{3}-\frac{4}{3}\left(2p^{\frac{1}{2}}+1\right)^3\left(p^{\frac{1}{2}}+1\right)\right).\end{align}

Let $\{ x_1, x_2, \ldots, x_{\frac{p-1}{3}}\}$ be the set of cubic residues modulo $p$. If $x_{i_1} + x_{i_2} + \cdots + x_{i_9}  \equiv n \pmod p$, then we define a vector
$$v_{i_1, i_2, \ldots, i_9} = e_{i_1} + e_{i_2} + \cdots + e_{i_9}$$
in $\mathbb{R}^{\frac{p-1}{3}}$ to be the corresponding vector to $(x_{i_1}, x_{i_2}, \ldots, x_{i_9} )$,
where $e_k$ is the vector with $k-th$ coordinate 1 and
other coordinates 0. Let
$$N = \{v_{i_1, i_2, \ldots, i_9} : x_{i_1} + x_{i_2} + \cdots + x_{i_9}  \equiv n \pmod p\}$$
be the set of all corresponding vectors. For any vector $u$ in $\mathbb{R}^{\frac{p-1}{3}}$, we write $u_i$ for its $i-th$ component.
\begin{lemma}\label{lemma1}
The set $N$ spans the whole vector space $\mathbb{R}^{\frac{p-1}{3}}$.
\end{lemma}
\begin{proof}
Let $M$ be the space spanned by $N$. For a fixed vector $u \in \mathbb{R}^{\frac{p-1}{3}}$, define a function $S$ on $M$ by
$$S(w) = (w_1 - u_1)^2 + \cdots + (w_{\frac{p-1}{3}} - u_{\frac{p-1}{3}})^2.$$
Since $M$ is a closed subspace, there exists a point $v \in M$ such that $S(v) = \min_{w\in M} S(w)$. We shall now prove $S(v) = 0$, which is equivalent to proving $v=u$ and $u\in M$. Assuming the contrary that $S(v) > 0$. Let $A_{+} = \{i : v_i > u_i\}$, $A_{-} = \{i : v_i < u_i\}$ and $A_0 = \{i : v_i = u_i\}$. There are two cases:

Case 1: $|A_{+}| + |A_{-}| = 1$. Without loss of generality, we assume that $|A_{+}| = 1$, $|A_{-}| = 0$ and $A_{+} = \{j_0\}$. Then there
exist $j_6, j_7, j_8, j_9$ such that $5x_{j_0} + x_{j_6}+\cdots + x_{j_9} \equiv n \pmod p$, by the Cauchy-Davenport theorem from \cite[Theorem 5.4]{tao}. Thus, if $\epsilon$ is small enough, then we have
$$S(v - \epsilon v_{j_0, j_0, j_0, j_0, j_0, j_6, j_7, j_8, j_9}) = S(v) - 10\epsilon (v_{j_0} - u_{j_0} ) + 29\epsilon ^2 < S(v).$$
Since $v - \epsilon v_{j_0, j_0, j_0, j_0, j_0, j_6, j_7, j_8, j_9}\in M,$ this is a contradiction with the assumption that $S(v)$ is minimal. Therefore $S(v) = 0$.

Case 2: $|A_{+}|+|A_{-}| \geq 2$. Then we let $j_{0}, j_{1} \in A_{+} \cup A_{-}$ such that $|v_{j_{0}}-u_{j_{0}}|=\max _{i \in A_{+} \cup A_{-}}|v_{i}-u_{i}|$ and $|v_{j_{1}}-u_{j_{1}}|=\max _{i \in A_{+} \cup A_{-}\backslash \{j_{0}\}}|v_{i}-u_{i}|$.

(i)\ If $n \not\equiv 6 x_{j_{0}} \pmod p$, then there exist $j_{7}, j_{8}, j_{9}$ such that $6 x_{j_{0}}+x_{j_{7}}+x_{j_{8}}+x_{j_{9}} \equiv n\pmod p$, by (\ref{hnn}). Hence $v-\epsilon v_{j_{0}, \ldots, j_{0}, j_{7}, j_{8},  j_{9}} \in M$. Without loss of generality, we assume that $v_{j_{0}}>u_{j_{0}}$, then given the condition that $\epsilon$ is small enough, one gets
$$S(v-\epsilon v_{j_{0}, \ldots, j_{0}, j_{7}, j_{8},  j_{9}}) \leq S(v)-12 \epsilon(v_{j_{0}}-u_{j_{0}})+2 \epsilon(|v_{j_{7}}-u_{j_{7}}|+\cdots+|v_{j_{9}}-u_{j_{9}}|)+39 \epsilon^{2},$$
which is smaller than $S(v)$, leading to a contradiction to the assumption that $S(v)$ is minimal. Therefore $S(v) = 0$.

(ii)\ If $n \equiv 6 x_{j_{0}}\pmod p$, then $n \not\equiv 6 x_{j_{1}} \pmod p$. There exist $j_{7}, j_{8}, j_{9}$ such that $j_{7}, j_{8}, j_{9} \neq j_{0}$ and $6 x_{j_{1}}+x_{j_{7}}+x_{j_{8}}+x_{j_{9}} \equiv n\pmod p$, by (\ref{hnn}). Repeat the above process and we obtain the contradiction. Therefore $S(v) = 0$.

From the above argument, we have $u \in M$, $\mathbb{R}^{\frac{p-1}{3}}\subseteq M$. Hence, $M=\mathbb{R}^{\frac{p-1}{3}}$.
\end{proof}
\begin{lemma}\label{lemma2}
Define the set $N' = \{\sum_{\alpha} \delta_{\alpha}v_{\alpha}:\delta_{\alpha}\geq0,v_{\alpha}\in N\} \cap [0, 1]^{\frac{p-1}{3}}$, where $\{\sum_{\alpha} \delta_{\alpha}v_{\alpha}:\delta_{\alpha}\geq0,v_{\alpha}\in N\}$ is the set of all non-negative linear combinations of the vectors in $N$. Then $(1, 1, . . . , 1) \in N'$.
\end{lemma}
\begin{proof}
One can easily check that $N'$ is a compact set. Assuming that $(1,1, ..., 1) \notin N'$. We apply the same method as in Lemma \ref{lemma1} and define a function $S$ on $N'$ by $S(u)=(1-u_{1})^{2}+\cdots+(1-u_{\frac{p-1}{3}})^{2}$. Then there exists a point $v \in N'$ such that $S(v)$ attains its minimum on $N'$. Let $A_{0}=\{i: v_{i}=1\}$ and $A_{1}=\{i: v_{i}<1\}$. There are following two cases:

Case 1: $|A_{1}|=1$. Suppose that $A_{1}=\{i_{0}\}$. Next, we proceed by dividing this situation into the following two cases.

(i)\ If $6 x_{i_{0}} \not\equiv n \pmod p$, then there exists $j_{7}, j_{8}, j_{9}$ which are not all equal and $6 x_{i_{0}}+ x_{j_{7}} + x_{j_{8}}+x_{j_{9}} \equiv n\pmod p$, by (\ref{hnn}). Thus the vector $\frac{1}{1+2\epsilon}(v+\epsilon v_{i_{0}, \ldots, i_{0}, j_{7}, j_{8}, j_{9}})$ belongs to $N'$ and we have
\begin{align*} S\left(\frac{1}{1+2\epsilon}(v+\epsilon v_{i_{0}, \ldots, i_{0}, j_{7}, j_{8}, j_{9}})\right) & \leq \frac{p-1}{3}\left(\frac{2\epsilon}{1+2\epsilon}\right)^{2}+\left(1-\frac{v_{i_{0}}+6 \epsilon}{1+2\epsilon}\right)^{2} \\ & <\left(1-v_{i_{0}}\right)^{2}=S(v) \end{align*}
given the condition that $\epsilon$ is sufficiently small. This contradicts the minimality of $S(v)$. Then, we have $(1,1, ..., 1) \in N'$.

(ii)\ If $6 x_{i_{0}} \equiv n\pmod p$, then there exist $j_{6}, j_{7}, j_{8}, j_{9}$ which are not all equal and such that $5 x_{i_{0}}+x_{j_{6}}+x_{j_{7}}+x_{j_{8}}+x_{j_{9}} \equiv n \pmod p$, by the Cauchy-Davenport theorem from \cite[Theorem 5.4]{tao}. Then we have $\frac{1}{1+3 \epsilon}(v+\epsilon v_{i_{0}, \ldots, i_{0}, j_{6}, j_{7}, j_{8}, j_{9}}) \in N'$, and
\begin{align*} S\left(\frac{1}{1+3 \epsilon}(v+\epsilon v_{i_{0}, \ldots, i_{0}, j_{6}, j_{7}, j_{8}, j_{9}})\right) & \leq \frac{p-1}{3}\left(\frac{3 \epsilon}{1+3\epsilon}\right)^{2}+\left(1-\frac{v_{i_{0}}+5 \epsilon}{1+3 \epsilon}\right)^{2} \\ & <\left(1-v_{i_{0}}\right)^{2}=S(v) \end{align*}
for sufficiently small $\epsilon$, which contradicts our assumption. Then, we have $(1,1, ..., 1) \in N'$.

Case 2: $|A_{1}| \geq2$. Suppose $|A_{1}|=l\geq2$. Next, we proceed by dividing this situation into the following two cases.

(i)\ If $l\geq\frac{p+8}{9}$, by the Cauchy-Davenport theorem \cite[Theorem 5.4]{tao}, we have
\begin{align*}A_{1}+A_{1}+A_{1}+A_{1}+A_{1}+A_{1}+A_{1}+A_{1}+A_{1}=\mathbb{Z}_{p}.\end{align*}
Thus we can find $i_{1}, i_{2}, \ldots, i_{9} \in A_{1}$ such that $x_{i_{1}}+x_{i_{2}}+\cdots+x_{i_{9}} \equiv n\pmod p$ and the vector $v+\epsilon v_{i_{1}, i_{2},\ldots, i_{9}}$ belongs to $N'$ for sufficiently small $\epsilon$ and clearly $S(v+\epsilon v_{i_{1}, i_{2},\ldots, i_{9}})<S(v)$, for sufficiently small $\epsilon$, which contradicts our assumption.

(ii)\ We assume that $l<\frac{p+8}{9}$. Write $A_{1}=\{i_{1}, ..., i_{l}\}$, and assume that $1-v_{i_{1}} \geq\cdots \geq1-v_{i_{l}}>0$.

If $7 x_{i_{1}} \not\equiv n \pmod p$. Consider the solutions to the equation $x_{j_{1}}+x_{j_{2}} \equiv n-7 x_{i_{1}}\pmod p$ which satisfies $j_{1}<j_{2}$. Let $s$ be the number of these solutions. Then any $1 \leq j\leq \frac{p-1}{3}$ appears at most one time in these $s$ solutions. Thus, we get
$$v'=\frac{1}{1+\epsilon}\left(v+\sum_{\substack{j_{1}<j_{2} \\ x_{j_{1}}+x_{j_{2}} \equiv n-7 x_{i_{1}}\pmod p}} \epsilon v_{i_{1}, i_{1}, i_{1}, i_{1}, i_{1}, i_{1}, i_{1}, j_{1}, j_{2}}\right) \in N' .$$
Moreover, we have
\begin{align*} S(v') & \leq \frac{p-1}{3}\left(\frac{\epsilon}{1+\epsilon}\right)^{2}+\left(1-\frac{v_{i_{1}}+7 s \epsilon}{1+\epsilon}\right)^{2}+\left(1-\frac{v_{i_{2}}}{1+\epsilon}\right)^{2}+\cdots+\left(1-\frac{v_{i_{l}}}{1+\epsilon}\right)^{2} \\ & =\sum_{k=1}^{l}\left(\frac{1-v_{i_{k}}}{1+\epsilon}\right)^{2}-\frac{2(7 s-1) \epsilon(1-v_{i_{1}})}{(1+\epsilon)^{2}}+\sum_{k=2}^{l} \frac{2 \epsilon\left(1-v_{i_{k}}\right)}{(1+\epsilon)^{2}}+O\left(\epsilon^{2}\right) \\ & \leq \sum_{k=1}^{l}\left(1-v_{i_{k}}\right)^{2}-\frac{2 \epsilon}{(1+\epsilon)^{2}}\left((7 s-1)(1-v_{i_{1}})-\sum_{k=2}^{l}\left(1-v_{i_{k}}\right)\right)+O\left(\epsilon^{2}\right) \\ & \leq S(v)-\frac{2 \epsilon}{(1+\epsilon)^{2}}(7 s-l)\left(1-v_{i_{1}}\right)+O\left(\epsilon^{2}\right) . \end{align*}
If $7 s>l$, for sufficiently small $\epsilon$, we have $S(v')<S(v)$, leading to a contradiction. Then we have $(1,1, ..., 1) \in N'$.

On the other hand, if $7 x_{i_{1}} \equiv n \pmod p$, then $6 x_{i_{1}}+x_{i_{2}} \not\equiv n \pmod p$. Let $t$ be the number of solutions to $x_{j_{1}}+x_{j_{2}} \equiv n-6 x_{i_{1}}-x_{i_{2}} \pmod p$ that satisfies $j_{1}<j_{2}$. Again each $1 \leq j \leq\frac{p-1}{3}$ appears at most once in those $t$ solutions. Hence
$$v'=\frac{1}{1+\epsilon}\left(v+\sum_{\substack{j_{1}<j_{2} \\ x_{j_{1}}+x_{j_{2}} \equiv n-6 x_{i_{1}}-x_{i_{2}}\pmod p}} \epsilon v_{i_{1}, i_{1}, i_{1}, i_{1}, i_{1}, i_{1}, i_{2}, j_{1}, j_{2}}\right) \in N',$$
and we have
\begin{align*} S\left(v'\right) \leq & \frac{p-1}{3}\left(\frac{\epsilon}{1+\epsilon}\right)^{2}+\left(1-\frac{v_{i_{1}}+6 t \epsilon}{1+\epsilon}\right)^{2}+\left(1-\frac{v_{i_{2}}+t \epsilon}{1+\epsilon}\right)^{2}+\left(1-\frac{v_{i_{3}}}{1+\epsilon}\right)^{2}+\cdots\\
&+\left(1-\frac{v_{i_{l}}}{1+\epsilon}\right)^{2} \\
 =&\sum_{k=1}^{l}\left(\frac{1-v_{i_{k}}}{1+\epsilon}\right)^{2}-\frac{2(6 t-1) \epsilon\left(1-v_{i_{1}}\right)}{(1+\epsilon)^{2}}-\frac{2(t-1) \epsilon\left(1-v_{i_{2}}\right)}{(1+\epsilon)^{2}}+\sum_{k=3}^{l} \frac{2 \epsilon\left(1-v_{i_{k}}\right)}{(1+\epsilon)^{2}}\\
 &+O\left(\epsilon^{2}\right) \\
\leq &
\sum_{k=1}^{l}(1-v_{i_{k}})^{2}-\frac{2\epsilon}{(1+\epsilon)^{2}}\left((6 t-1)(1-v_{i_{1}})+(t-1) (1-v_{i_{2}})-\sum_{k=3}^{l} (1-v_{i_{k}})\right)\\
&+O(\epsilon^{2})\\
\leq &
S(v)-\frac{2 \epsilon}{(1+\epsilon)^{2}}(7 t-l)\left(1-v_{i_{2}}\right)+O(\epsilon^{2}).
\end{align*}
Similarly, if $7 t>l$, we have $S(v')<S(v)$ for sufficiently small $\epsilon$, which is a contradiction. Then we have $(1,1, ..., 1) \in N'$.

Next we only need to prove $s,t>\frac{l}{7}$. By our definition of $s$, $t$ and $h(n)$, we have $s, t \geq\frac{h(n)}{18}-1 \geq\frac{p+8}{63}>\frac{l}{7}$ for $p \geq139$.

Consider the equation $x_1^3 + x_2^3 + \dots + x_9^3 \equiv n \pmod p$. Multiplying both sides by a non-zero cubic residue $r^3$ leaves the number of solutions unchanged, that is, the equations corresponding to $n$ and $r^3n$ possess an identical number of solutions. Utilizing this property, we can partition the non-zero elements $1, 2, \ldots, p-1$ into distinct equivalence classes. Consequently, any target values $n$ belonging to the same class yield exactly the same number of solutions.

For $p=127$, one has that $l \leq14$ and $s, t \geq3$. In fact, for modulo $127$, the $126$ non-zero elements are partitioned into exactly $3$ distinct classes regarding cubic residues. For example, $3$ and $14$ belong to the same class, because $14\equiv47\times3\pmod{127}$, where 47 is a cubic residue. These 3 classes of sets are respectively:
\begin{align*}
B_1 &= \{1, 2, 4, 5, 8, 10, 16, 19, 20, 25, 27, 32, 33, 38, \\
&\qquad 40, 47, 50, 51, 54, 61, 63, 64, 66, 73, 76, 77, 80, 87, \\
&\qquad 89, 94, 95, 100, 102, 107, 108, 111, 117, 119, 122, 123, 125, 126\}, \\[2ex]
B_2 &= \{3, 6, 7, 12, 13, 14, 15, 23, 24, 26, 28, 30, 31, 35, \\
&\qquad 46, 48, 52, 56, 57, 60, 62, 65, 67, 70, 71, 75, 79, 81, \\
&\qquad 92, 96, 97, 99, 101, 103, 104, 112, 113, 114, 115, 120, 121, 124\}, \\[2ex]
B_3 &= \{9, 11, 17, 18, 21, 22, 29, 34, 36, 37, 39, 41, 42, 43, \\
&\qquad 44, 45, 49, 53, 55, 58, 59, 68, 69, 72, 74, 78, 82, 83, \\
&\qquad 84, 85, 86, 88, 90, 91, 93, 98, 105, 106, 109, 110, 116, 118\}.
\end{align*}
Therefore, to prove that the theorem holds for all 126 non-zero values of $n$, we do not need to check every single one, it suffices to select and verify the simplest number from each of the 3 classes mentioned above.
We choose $n=1$ represents the class of cubic residues, $n=3$ represents the first non-residue class and $n=9$ naturally falls into the second non-residue class. So we have $1 \equiv 2+126\equiv 5+123\equiv 20+108 \pmod {127}$, $3 \equiv 1+2\equiv 4+126\equiv 5+125 \pmod {127}$ and $9 \equiv 1+8\equiv 4+5\equiv 10+126 \pmod {127}$, where $1, 2, 4, 5, 8, 10, 16, 19, 20, 25, 27, 32, 33, 38, 40, 47, 50, 51, 54, 61, 63, 64, 66, 73, 76, 77,80, 87, 89,$\\ $94, 95, 100, 102, 107, 108, 111, 117, 119, 122, 123, 125, 126$ are cubic residues.

For $p=109$, one has that $l \leq12$ and $s, t \geq2$. In fact, we have $1 \equiv 2+108\equiv 46+64\equiv 33+77 \pmod {109}$, $3 \equiv 1+2\equiv 4+108\equiv 46+66 \pmod {109}$ and $9 \equiv 1+8\equiv 64+54\equiv 17+101 \pmod {109}$, where $1, 2, 4, 8, 16, 17, 19, 23, 27, 32, 33, 34, 38, 41, 43, 45, 46, 54, 55, 63, 64, 66, 68,$\\ $71, 75, 76, 77, 82, 86, 90, 92, 93, 101, 105, 107, 108$ are cubic residues.

For $p=103$, one has that $l \leq12$ and $s, t \geq2$. In fact, we have $1 \equiv 9+95\equiv 10+94\equiv 14+90 \pmod {103}$, $2 \equiv 10+95\equiv 24+81\equiv 39+66 \pmod {103}$ and $4 \equiv 1+3\equiv 14+93\equiv 27+80 \pmod {103}$, where $1, 3, 8, 9, 10, 13, 14, 22, 23, 24, 27, 30, 31, 34, 37, 39, 42, 61, 64, 66, 69, 72, 73,\\ 76, 79, 80, 81, 89, 90, 93, 94, 95, 100, 102$ are cubic residues.

Similarly, for $p=97$ one has $l \leq11$ and $s, t \geq2$. For $p=79$ we have $l \leq9$ and $s, t \geq2$. For $p=73$ we have $l \leq8$ and $s, t \geq2$. For $p=67$ we have $l \leq8$ and $s, t \geq2$. For $p=61$ we have $l \leq7$ and $s, t \geq2$. For $p=43$ we have $l \leq5$ and $s, t \geq1$. For $p=37$ we have $l \leq4$ and $s, t \geq1$. For $p=31$ we have $l \leq4$ and $s, t \geq1$. Thus $s, t>\frac{l}{7}$ for all $p \geq31, p\equiv1\pmod3$.

We remark that the above two lemmas also hold for $p=13,19$, which can be calculated directly and we list them below. We only need to list them for the there cases $n=1$, $n=2$ and $n=4$.

For $p=13$, $n=1$, set $(x_{1}, x_{2}, x_{3}, x_{4})=(1,5,8,12)$. Consider the vectors
$$(5, 0, 0, 4), (2, 6, 1, 0), (0, 2, 7, 0), (0, 0, 4, 5),$$
for $p=13$, $n=2$, consider the vectors
$$(0, 4, 0, 5), (4, 1, 4, 0), (5, 0, 3, 1), (1, 2, 4, 2),$$
and for $p=13$, $n=4$, consider the vectors
$$(7, 2, 0, 0), (0, 5, 1, 3), (2, 2, 4, 1), (0, 0, 0, 9).$$
The three sets of vectors listed above respectively form a basis for $\mathbb{R}^4$, and the coordinates of $(1, 1, 1, 1)$ with respect to these three bases are all positive.

For $p=19$, $n=1$, set $(x_{1}, x_{2}, x_{3}, x_{4}, x_{5}, x_{6})=(1,7,8,11,12,18)$. Apply the vectors
$$(5, 0, 0, 0, 0, 4), (1, 7, 1, 0, 0, 0), (2, 0, 7, 0, 0, 0), (1, 0, 0, 7, 0, 1), (1, 0, 0, 1, 7, 0), (1, 1, 0, 0, 0, 7),$$
for $p=19$, $n=2$, apply
$$(7, 2, 0, 0, 0, 0), (0, 6, 0, 0, 3, 0), (0, 0, 7, 2, 0, 0), (4, 0, 0, 5, 0, 0), (1, 0, 0, 0, 8, 0), (1, 0, 1, 0, 0, 7),$$
and for $p=19$, $n=4$, apply
$$(7, 0, 2, 0, 0, 0), (1, 7, 0, 1, 0, 0), (2, 0, 6, 1, 0, 0), (0, 0, 0, 9, 0, 0), (0, 1, 1, 0, 7, 0), (0, 0, 0, 0, 1, 8).$$
The three sets of vectors listed above respectively form a basis for $\mathbb{R}^6$, and the coordinates of $(1, 1, 1, 1, 1, 1)$ with respect to these three bases are all positive.
\end{proof}

For $p\geq5$, $p\equiv 2\pmod 3$,  by the inequality
$$\left|\sum_{x=0}^{p-1} e\left(\frac{r x^{3}}{p}\right)\right|=0$$
for $1 \leq r \leq p-1$, we have
$$\sum_{\left(\frac{r}{p}\right)_3=1} e\left(\frac{r}{p}\right)=\sum_{r=1}^{p-1} e\left(\frac{r}{p}\right)=-1\quad \text{and}\quad \sum_{\left(\frac{r}{p}\right)_3=-1} e\left(\frac{r}{p}\right)=0.$$
Finally, we have
\begin{align*}h(n)\geq p-2.\end{align*}
Similarly, we have
\begin{align*}h'(n)\geq \frac{1}{p}\left((p-1)^3-1\right).\end{align*}

Let $\{ x_1, x_2, \ldots, x_{p-1}\}$ be the set of cubic residues modulo $p$. If $x_{i_1} + x_{i_2} + \cdots + x_{i_9}  \equiv n \pmod p$, then we define a vector
$$v^{\ast}_{i_1, i_2, \ldots, i_9} = e_{i_1} + e_{i_2} + \cdots + e_{i_9}$$
in $\mathbb{R}^{p-1}$ to be the corresponding vector to $(x_{i_1}, x_{i_2}, \ldots, x_{i_9} )$,
where $e_k$ is the vector with $k-th$ coordinate 1 and
other coordinates 0. Let
$$N^{\ast} = \{v^{\ast}_{i_1, i_2, \ldots, i_9} : x_{i_1} + x_{i_2} + \cdots + x_{i_9}  \equiv n \pmod p\}$$
be the set of all corresponding vectors. For any vector $u$ in $\mathbb{R}^{p-1}$, we write $u_i$ for its $i-th$ component.

From the above two Lemmas, we can obtain the following two corollaries, whose proofs are similar to those of Lemma \ref{lemma1} and Lemma \ref{lemma2} respectively.
\begin{corollary}\label{co1}
The set $N^{\ast}$ spans the whole vector space $\mathbb{R}^{p-1}$.
\end{corollary}
\begin{corollary}\label{co2}
Define the set $N'^{\ast}= \{\sum_{\alpha} \delta_{\alpha}v_{\alpha}:\delta_{\alpha}\geq0,v_{\alpha}\in N^{\ast}\} \cap [0, 1]^{p-1}$, where $\{\sum_{\alpha} \delta_{\alpha}v_{\alpha}:\delta_{\alpha}\geq0,v_{\alpha}\in N^{\ast}\}$ is the set of all non-negative linear combinations of the vectors in $N^{\ast}$. Then $(1, 1, . . . , 1) \in N'^{\ast}$.
\end{corollary}
\begin{xremark}
When $p=3$, it is easily verified that Corollary \ref{co1} and Corollary \ref{co2} hold.
\end{xremark}

\section{Proof of Theorem 1.3}
In this section, we will use Lemma \ref{lemma2} and Corollary \ref{co2} to prove Theorem \ref{THM1}. Then, we will give the Corollary \ref{3.1} by Theorem \ref{THM1}, to prove Theorem \ref{THM2}.

\noindent $\textit{Proof of Theorem \ref{THM1}}$\\
Let
$$K=\frac{1}{\phi_{3}(q)} \sum_{a \in \mathbb{Z}_{q}^{(3)}} \sum_{i=1}^{9} f_{i}(a).$$
Firstly we will prove Theorem \ref{THM1}, when $q = p \geq 11$ is a prime and $p \equiv 1\pmod 3$. Assuming the contrary that there exists $n \in \mathbb{Z}_{p}$ such that for any $x_{1}, x_{2}, \ldots, x_{9} \in \mathbb{Z}_{p}^{(3)}$ with $n\equiv x_{1}+x_{2}+\cdots+x_{9}\pmod {p}$, one has $\sum_{i=1}^{9} f_{i}(x_{i})<K$. Suppose that $\mathbb{Z}_{p}^{(3)}=\{y_{1}, ..., y_{\frac{p-1}{3}}\}$.

If $n=0$, then for any $1 \leq k_{1}, k_{2}, \ldots, k_{9} \leq \frac{p-1}{3}$ such that $y_{k_{1}}+y_{k_{2}}+\cdots+y_{k_{9}}\equiv0\pmod p$, we have
$$y_{j} y_{k_{1}}+y_{j} y_{k_{2}}+\cdots+y_{j} y_{k_{9}}\equiv0\pmod p$$
for all $j$. Hence we have
$$\sum_{j=1}^{\frac{p-1}{3}} \sum_{\sigma \in S_{9}} \sum_{i=1}^{9} f_{i}\left(y_{j} y_{k_{\sigma(i)}}\right)=9! \sum_{a \in \mathbb{Z}_{p}^{(3)}} \sum_{i=1}^{9} f_{i}(a)=\frac{9!}{3}(p-1)K.$$
However, by our assumption we have
$$\sum_{j=1}^{\frac{p-1}{3}} \sum_{\sigma \in S_{9}} \sum_{i=1}^{9} f_{i}\left(y_{j} y_{k_{\sigma(i)}}\right)<\frac{9!}{3}(p-1) K,$$
which is a contradiction.

For $n \neq0$, by Lemma \ref{lemma2}, we obtain a set $\{k_{i}^{\alpha}: 1\leq i \leq9\}$ such that $\sum_{i=1}^{9} y_{k_{i}^{\alpha}}\equiv n\pmod p$ for all $\alpha$ and $\delta_{\alpha} \geq0$ such that
$$\sum_{\alpha} \delta_{\alpha} \sum_{i=1}^{9} e_{k_{i}^{\alpha}}=\sum_{j=1}^{\frac{p-1}{3}} e_{j}.$$
Therefore, $\sum_{\alpha} \delta_{\alpha}=\frac{p-1}{27}$. Then, we have
\begin{align*}\sum_{\alpha} \delta_{\alpha} \sum_{\sigma \in S_{9}} \sum_{i=1}^{9} f_{i}\left(y_{k_{\sigma(i)}^{\alpha}}\right)&=8! \sum_{\alpha} \delta_{\alpha} \sum_{i=1}^9 \sum_{j=1}^9 f_i(y_{k_j^\alpha})\\
 &=8! \sum_{a \in \mathbb{Z}_p^{(3)}} \sum_{i=1}^9 f_i(a)\\
&=\frac{8!}{3}(p-1)K.\end{align*}
However, we have
$$\sum_{\alpha} \delta_{\alpha} \sum_{\sigma \in S_{9}} \sum_{i=1}^{9} f_{i}\left(y_{k_{\sigma(i)}^{\alpha}}\right)<\frac{8!}{3}(p-1) K$$
since $\sum_{i=1}^{9} y_{k_{\sigma(i)}^{\alpha}}\equiv n\pmod p$ for all $\alpha$ and $\sigma \in S_{9}$. This leads to a contradiction.

In case $p\geq5, p\equiv2\pmod3$ and in case $p=3$, by corollary \ref{co2}, the proof  is similar to that of in case $p\geq11$, $p\equiv1\pmod3$. Therefore, Theorem \ref{THM1} holds under conditions $p=3$, $p=5$ or $p\geq11$.

Next we will prove that Theorem \ref{THM1} holds for $q = p_1p_2$ with $p_{1}\neq p_{2}$. Let $p_{1}' \in \mathbb{Z}_{q}$ such that $p_{1} p_{1}' \equiv 1 \pmod
{p_{2}}$. Similarly let $p_{2}'$ be the reciprocal of $p_{2}$ modulo $p_{1}$. Then consider the homomorphism $\iota : \mathbb{Z}_{p_{1}} \oplus \mathbb{Z}_{p_{2}} \longrightarrow \mathbb{Z}_{q}$ defined by $\iota(x_{1}, x_{2})=x_{1} p_{2} p_{2}'+x_{2} p_{1} p_{1}'$. This is a group isomorphism, and since $\iota(x_{1}, x_{2}) \iota(y_{1}, y_{2})\equiv\iota(x_{1} y_{1}, x_{2} y_{2})\pmod q$, one can see $\iota$ is also an isomorphism between $\mathbb{Z}_{p_{1}}^{(3)} \oplus \mathbb{Z}_{p_{2}}^{(3)}$ and $\mathbb{Z}_{q}^{(3)}$ and so $\phi_{3}(q)=\phi_{3}(p_{1}) \phi_{3}(p_{2})$. We identify $f_{i}$, $i=1, 2, 3, \ldots, 9$, as functions on $\mathbb{Z}_{p_{1}}^{(3)} \oplus \mathbb{Z}_{p_{2}}^{(3)}$. For any $n \in \mathbb{Z}_{q}$, there exists $n_{1} \in \mathbb{Z}_{p_{1}}$, $n_{2} \in \mathbb{Z}_{p_{2}}$ such that $n=\iota(n_{1}, n_{2})$. Now we define $g_{i}$ on $\mathbb{Z}_{p_{1}}^{(3)}$ by
$$g_{i}(a)=\frac{1}{\phi_{3}\left(p_{2}\right)} \sum_{b \in \mathbb{Z}_{p_{2}}^{(3)}} f_{i}(a, b)$$
for $i=1, 2, 3, \ldots, 9$. Then by our assumption, there exist $x_{1}, x_{2}, \ldots, x_{9} \in \mathbb{Z}_{p_{1}}^{(3)}$ such that $n_{1}\equiv x_{1}+x_{2}+\cdots+x_{9}\pmod {p_1}$ and
$$\sum_{i=1}^{9} g_{i}\left(x_{i}\right) \geq \frac{1}{\phi_{3}\left(p_{1}\right)} \sum_{a \in \mathbb{Z}_{p_{1}}^{(3)}} \sum_{i=1}^{9} g_{i}(a)=K_1.$$
Thus, we have
$$\frac{1}{\phi_{3}\left(p_{2}\right)} \sum_{b \in \mathbb{Z}_{p_{2}}^{(3)}} \sum_{i=1}^{9} f_{i}\left(x_{i}, b\right) \geq K_1.$$
Now apply our hypothesis to the functions $f_{i}(x_{i}, b)$ on $\mathbb{Z}_{p_{2}}^{(3)}$, we obtain $y_{1}, y_{2}, \ldots, y_{9} \in \mathbb{Z}_{p_{2}}^{(3)}$ such that $n_{2}\equiv y_{1}+y_{2}+\cdots+y_{9}\pmod {p_2}$, and
\begin{align*}\sum_{i=1}^{9} f_{i}\left(x_{i}, y_{i}\right)& \geq \frac{1}{\phi_{3}\left(p_{2}\right)} \sum_{b \in \mathbb{Z}_{p_{2}}^{(3)}} \sum_{i=1}^{9} f_{i}\left(x_{i}, b\right) \\ &\geq K_1\\ &=
\frac{1}{\phi_{3}\left(p_{2}\right)}\frac{1}{\phi_{3}\left(p_{1}\right)} \sum_{b \in \mathbb{Z}_{p_{2}}^{(3)}}\sum_{a \in \mathbb{Z}_{p_{1}}^{(3)}} \sum_{i=1}^{9} f_{i}\left(a, b\right)\\ &=K.\end{align*}

Now, we can use the similar method to prove that Theorem \ref{THM1} holds for $q = p_1p_2\cdots p_s$ with $p_{1}, p_{2},\ldots,p_s$ are different primes.
Thus, we complete the proof of Theorem \ref{THM1}.
\begin{corollary}\label{3.1}
Let $q$ be a square-free integer with $(q, 2)=1$. Suppose that $f_{i}: \mathbb{Z}_{q}^{(3)} \longrightarrow [0,1]$, $1 \leq i \leq 9$, are nine functions satisfying that
\begin{align}\label{SS1}\sum_{a \in \mathbb{Z}_{q}^{(3)}} \sum_{i=1}^{9} f_{i}(a)>8 \phi_{3}(q).\end{align}
Then, for any $n \in \mathbb{Z}_{q},$ there exist $x_{1}, x_{2}, \ldots x_{9} \in \mathbb{Z}_{q}^{(3)}$ such that $n\equiv x_{1}+x_{2}+\cdots+x_{9}\pmod {q}$ and $\sum_{i=1}^{9} f_{i}(x_{i})>\frac{67}{9}$ and $f_{i}(x_{i})>0$, $1\leq i\leq9$.
\end{corollary}
\begin{proof}
For $q=7$, the $7$ elements are partitioned into exactly 4 distinct classes regarding cubic residues. For example, $2$ and $5$ belong to the same class, because $5\equiv6\times2\pmod{7}$, where 6 is a cubic residue. These 4 classes of sets are respectively:
\begin{align*}
    D_1 &= \{0\}, \\
    D_2 &= \{1, 6\},\\
    D_3 &= \{2, 5\},\\
    D_4 &= \{3, 4\}, \\
\end{align*}
Therefore, to prove that the theorem holds for all 7 values of $n$, we do not need to check every single one, it suffices to select and verify the simplest number from each of the 7 classes mentioned above. We only need to check $n=0, 1, 2, 3$.

If $n=0$, then either $\sum_{i=1}^{9} f_{i}(1)>8$ or $\sum_{i=1}^{9} f_{i}(6)>8$ and we get the conclusion.

If $n=1$, then
\begin{align*}
\max_{\sigma \in S_9} \left\{ \sum_{i=1}^5 f_{\sigma(i)}(1) + \sum_{i=6}^9 f_{\sigma(i)}(6) \right\}
&\geq \frac{5}{9}[f_{\sigma(1)}(1)+\cdots+f_{\sigma(9)}(1)]+\frac{4}{9}[f_{\sigma(1)}(6)+\cdots+f_{\sigma(9)}(6)]\\
&>\frac{5}{9}\times8\times2 - \frac{1}{9}\times9 = \frac{71}{9}.
\end{align*}
We know that at most 1 entries in $\{f_i(a)|1<i<9,a\in\mathbb{Z}_{7}^{(3)}\}$ are zero by (\ref{SS1}). \\
If $f_1(1)=0$, then
$$\sum_{i=1}^4 f_{i}(6)+\sum_{i=5}^9 f_{i}(1)>16-\sum_{i=2}^4 f_{i}(1)-\sum_{i=5}^9 f_{i}(6)\geq8,$$
hence, $f_1(6), f_2(6), f_3(6), f_4(6), f_{5}(1), f_{6}(1), f_{7}(1), f_{8}(1), f_{9}(1)>0$.

If $n=2$, then
\begin{align*}
\max_{\sigma \in S_9} \left\{ \sum_{i=1}^2 f_{\sigma(i)}(1) + \sum_{i=3}^9 f_{\sigma(i)}(6) \right\}
> \frac{7}{9}\times8\times2 - \frac{5}{9}\times9 = \frac{67}{9}.
\end{align*}
If $f_1(1)=0$, then
$$\sum_{i=1}^7 f_{i}(6)+\sum_{i=8}^9 f_{i}(1)>16-\sum_{i=2}^7 f_{i}(1)+\sum_{i=8}^9 f_{i}(6)\geq8,$$
hence, $f_1(6), f_2(6), f_3(6), f_4(6), f_{5}(6), f_{6}(6), f_{7}(6), f_{8}(1), f_{9}(1)>0$.

If $n=3$, then
\begin{align*}
\max_{\sigma \in S_9} \left\{ \sum_{i=1}^6 f_{\sigma(i)}(1) + \sum_{i=7}^9 f_{\sigma(i)}(6) \right\}
> \frac{6}{9}\times8\times2 - \frac{3}{9}\times9 = \frac{69}{9}.
\end{align*}
The case $f_1(1) = 0$ is the same as above.

Now we consider the remaining case $q=7q'$, where $q'\in \mathbb{Z}_{\geq 2}$. Since $q$ be a square-free integer with $(q, 2)=1$, we have
$(q',7)=1$. We write $\mathbb{Z}_{q}=\mathbb{Z}_{7} \oplus \mathbb{Z}_{q'}$ as in Theorem \ref{THM1}. Suppose that $n=\iota(n_{1}, n_{2})$. We identify $f_{i}$, $i=1, 2, 3, \ldots, 9$, as functions on $\mathbb{Z}_{7}^{(3)} \oplus \mathbb{Z}_{q'}^{(3)}$. By the condition of Corollary \ref{3.1}, we have
$$\sum_{y \in \mathbb{Z}_{q'}^{(3)}} \sum_{b \in \mathbb{Z}_{7}^{(3)}} \sum_{i=1}^9 f_i(b, y) > 8 \times 2\phi_3(q') = 16\phi_3(q').$$
For $i=1, 2, 3, \ldots, 9$, define $g_{i}$ on $\mathbb{Z}_{q'}^{(3)}$ by
$$g_i(y) = \sum_{b \in \mathbb{Z}_{7}^{(3)}} f_i(b, y).$$
So we get$$\sum_{y \in \mathbb{Z}_{q'}^{(3)}} \sum_{i=1}^9 g_i(y) > 16\phi_3(q').$$
By Theorem \ref{THM1}, we get $y_{1}, y_{2}, y_{3}, y_{4}, y_{5}, y_{6}, y_{7}, y_{8}, y_{9} \in \mathbb{Z}_{q'}^{(3)}$ such that $n_{2}\equiv\sum_{i=1}^{9} y_{i}\pmod {q'}$ in $\mathbb{Z}_{q'}$, and
$$\sum_{i=1}^9 g_i(y_i) \ge \frac{1}{\phi_3(q')} \sum_{y \in \mathbb{Z}_{q'}^{(3)}} \sum_{i=1}^9 g_i(y).$$
Hence
$$\sum_{b \in\mathbb{Z}_{7}^{(3)}} \sum_{i=1}^{9}f_{i}\left(b, y_{i}\right)>16.$$
Repeat the above argument and we obtain $x_{1}, x_{2}, \ldots, x_{9} \in\mathbb{Z}_{7}^{(3)}$ such that $n_{1}\equiv\sum_{i=1}^{9} x_{i}\pmod {7}$ in $\mathbb{Z}_{7}$ and
$$\sum_{i=1}^{9} f_{i}\left(x_{i}, y_{i}\right)>\frac{67}{9}.$$
Moreover, we have $(x_{i}, y_{i}) \in \mathbb{Z}_{q}^{(3)}$ and $f_{i}(x_{i}, y_{i})>0$.
\end{proof}
\section{The transference principle}
In this section, we will reduce the problem of proving $n' \in A_1 + A_2 + \cdots +A_9$ to estimating the 9-fold convolution of the weighted indicator functions $a_i(x) = \mathds{1}_{A_i}(x)\lambda_i(x)$, by the transference principle of \cite{green}. Next we will estimate $\sum\limits_{\substack{x_{1}, x_{2}, \ldots, x_{9} \in \mathbb{Z}_{N} \\ x_{1}+x_{2}+\cdots+x_{9}=n'}} \prod_{i=1}^{9} a_{i}\left(x_{i}\right)$.

Let $\kappa=\frac{1}{10^{6}}(\sum_{i=1}^{9} \mu_{i}^{3}-8)$ and $\alpha_{i}=\frac{\mu_{i}}{1+2 \kappa}$. Let $n \equiv 1\pmod 2$ be a sufficiently large integer such that
$$\left|P_{i} \cap\left[1, \sqrt[3]{\frac{2}{9} n}\right]\right| \geq(1+\kappa) \alpha_{i} \frac{\sqrt[3]{\frac{2}{9} n}}{\frac{1}{3} \log n}.$$
Set $\omega=\omega_n=\frac{1}{100} \log \log n$ and
$$W=2\times \prod_{3 \leq p \leq \omega} p.$$
Assuming that $\delta_{1}, \delta_{2}$ are sufficiently small, we have
\begin{align}\label{4.1}\sum_{\substack{{x \leq \sqrt[3]{\frac{2}{9} n}}\\{(x,W)=1}}} 3 x^2 \log x \mathds{1}_{P_{i}}(x) & \geq \sum_{\omega<p \leq(1+\kappa) \alpha_{i} \sqrt[3]{\frac{2}{9} n}\left(1-\delta_{1}\right)} 3 p^2 \log p \notag \\ & \geq\left(1-\delta_{2}\right)\left(1-\delta_{1}\right)^{3}(1+\kappa)^{3} \alpha_{i}^{3} \frac{2}{9} n \geq \frac{2}{9} \alpha_{i}^{3} n,\end{align}
where we use the prime number theorem in the first inequality and the estimate
$$\sum_{A \leq p \leq B} 3 p^2 \log p=B^{3}+o\left(B^{3}\right)$$
as $B, \frac{B}{A} \longrightarrow \infty$ in the second inequality.
Now define
$$f_{i}(b)=\max \left\{0, \frac{9}{2} \frac{\phi_{3}(W)}{n} \sum_{\substack{x \leq \sqrt[3]{\frac{2}{9} n} \\ x^{3} \equiv b\pmod W}} 3 x^2 \log x \mathds{1}_{P_{i}}(x)-6 \kappa\right\}$$
for $b \in \mathbb{Z}_{W}^{(3)}$ and $1 \leq i \leq9$. Next, we will check that the $f_{i}(b)$ satisfies the condition of Corollary \ref{3.1}.
\begin{lemma}\label{lemma4.1}
Given the condition that $n$ is large enough, we have $f_{i}(b) \leq1$ for all $b \in \mathbb{Z}_{W}^{(3)}$ and $1 \leq i \leq9$.
\end{lemma}
\begin{proof}
Write $X=\sqrt[3]{\frac{2}{9} n}$. Then we have
\begin{align*}\sum_{\substack{p \leq \sqrt[3]{\frac{2}{9} n} \\ p^{3} \equiv b\pmod W}} 3 p^2 \log p \mathds{1}_{P_{i}}(p) &\leq \sum_{\substack{p \leq \sqrt[3]{\frac{2}{9} n} \\ p^{3} \equiv b\pmod W}} 3 p^2 \log p\\
&=\sum_{\substack{z \in[1, W] \\ z^{3} \equiv b\pmod W}}\sum_{\substack{p\leq X \\ p \equiv z\pmod W}} 3 p^2\log p.\end{align*}
For any $m \leq X$, put
$$S_{m}=\sum_{\substack{p \leq m \\ p \equiv z\pmod W}} 1.$$
Since $(z, W)=1$ and $W \leq \log X$, for $c_1$ is a constant, we have $$S_{m}=\frac{Li(m)}{\phi(W)}+O(X e^{-c_{1} \sqrt{\log X}})$$ by the Siegel–Walfisz theorem.
Write $g(n)=3n^2 \log n$. Then we get
\begin{align*}\sum_{\substack{p \leq X \\ p \equiv z\pmod W}} 3 p^2 \log p & =\sum_{k=2}^{X}\left(S_{k}-S_{k-1}\right) g(k) \\ & =S_{X} g(X)+\sum_{k=2}^{X} S_{k-1}(g(k-1)-g(k)) . \end{align*}
Notice that $g(k-1)-g(k)=O(X\log X)$ and by the definition of $S_X$, we have $S_1=0$. It follows that
\begin{align*} &\phi(W) \sum_{\substack{p \leq X \\ p\equiv z\pmod W}} 3 p^2 \log p \\
& =Li(X) g(X)+\sum_{k=3}^{X} Li(k-1)(g(k-1)-g(k))+O\left(X^{3} e^{-c_{2} \sqrt{\log X}}\right) \\ & =\sum_{k=3}^{X} g(k) \int_{k-1}^{k} \frac{d x}{\log x}+O\left(X^{3} e^{-c_{2} \sqrt{\log X}}\right) \\ & =\sum_{k=3}^{X} 3 k^2 \int_{k-1}^{k} \frac{\log k}{\log x} d x+O\left(X^{3} e^{-c_{2} \sqrt{\log X}}\right) \\ & =\sum_{k=3}^{X} 3k^2 \int_{k-1}^{k} \frac{\log x+O\left(\frac{1}{k}\right)}{\log x} d x+O\left(X^{3} e^{-c_{2} \sqrt{\log X}}\right) \\ & =X^{3}+O\left(X^{3} e^{-c_{2} \sqrt{\log X}}\right)\end{align*}
where $c_{2}$ is an absolute constant smaller than $c_{1}$.

For each $b \in \mathbb{Z}_{W}^{(3)}$, we write $\sigma_{3}(b)=\#\{z \in[1, W]: z^{3} \equiv b \pmod W\}$. One can see that $\phi(W)=\phi_{3}(W) \sigma_{3}(b)$ applying the Chinese Remainder Theorem.
If $f_{i}(b)=0$, we finish the proof. If $f_{i}(b)\neq0$, we obtain
\begin{align*}f_{i}(b) & \leq \frac{9}{2} \frac{\phi_{3}(W)}{n} \frac{\sigma_{3}(b)}{\phi(W)} X^{3}(1+o(1))-6 \kappa \\ & =1+o(1)-6 \kappa \leq 1\end{align*}
given the condition that $n$ is sufficiently large.
\end{proof}
From the Lemma \ref{lemma4.1}, we find that $f_{i}: \mathbb{Z}_{q}^{(3)} \longrightarrow [0,1]$, $1 \leq i \leq 9$, which is the part condition of Corollary \ref{3.1}. Next, we will prove
$$\sum_{b \in \mathbb{Z}_{W}^{(3)}} \sum_{i=1}^{9} f_{i}(b)>8 \phi_{3}(q).$$
By (\ref{4.1}), we have
\begin{align*}\sum_{b \in \mathbb{Z}_{W}^{(3)}} \sum_{i=1}^{9} f_{i}(b) & \geq \frac{9}{2} \frac{\phi_{3}(W)}{n} \sum_{b \in \mathbb{Z}_{W}^{(3)}} \sum_{\substack{x \leq \sqrt[3]{\frac{2}{9}n} \\ x^{3} \equiv b\pmod W}} \sum_{i=1}^{9} 3 x^2 \log x \mathds{1}_{P_{i}}(x)-54 \kappa \phi_{3}(W) \\ & =\frac{9}{2} \frac{\phi_{3}(W)}{n} \sum_{\substack{x \leq \sqrt[3]{\frac{2}{9} n} \\ (x, W)=1}} \sum_{i=1}^{9} 3 x^2 \log x \mathds{1}_{P_{i}}(x)-54 \kappa \phi_{3}(W) \\ & \geq \frac{9}{2} \frac{\phi_{3}(W)}{n} \sum_{i=1}^{9} \frac{2}{9} \alpha_{i}^{3} n-54 \kappa \phi_{3}(W) \\ & >8 \phi_{3}(W).
\end{align*}

Now, $f_i$ satisfies all the conditions of Corollary \ref{3.1}. Thus, by Corollary \ref{3.1}, we obtain $b_{1}, b_{2}, \ldots, b_{9} \in \mathbb{Z}_{W}^{(3)}$ such that $n \equiv b_{1}+b_{2}+\cdots+b_{9} \pmod W$. Moreover, we have $\sum_{i=1}^{9} f_{i}(b_{i})>\frac{67}{9}$ and $f_{i}(b_{i})>0$ for all $i$.

Assuming that $1 \leq b_{i}<W$. We choose $N$ to be a prime in $[\frac{(1+\kappa) n}{W}, \frac{(1+2 \kappa) n}{W}]$, whose existence is ensured by the prime number theorem, and let $n'=\frac{n-(\sum_{i=1}^{9} b_{i})}{W}$. Set $A_{i}=\{x: W x+b_{i}=p^{3}, p \in P_{i} \cap[1, \sqrt[3]{\frac{2}{9} n}]\}$. Then it suffices to show that $n' \in A_{1}+A_{2}+\cdots+A_{9}$.
Now we define
\begin{align*}\lambda_{i}(x)= \begin{cases}\phi_{3}(W) \frac{(W x+b_{i})^{\frac{2}{3}} \log \left(W x+b_{i}\right)}{W N}, & x \leq N,\ W x+b_{i}=p^{3} \ \text{for some}\ p \in \mathbb{P}, \\ 0, & \text{otherwise}. \end{cases}
\end{align*}
Set $\alpha_{i}'=\sum_{x} \mathds{1}_{A_{i}}(x) \lambda_{i}(x)$. Then
$$\alpha_{i}'=\sum_{\substack{p^{3} \equiv b_{i}\pmod W \\ p \leq \sqrt[3]{\frac{2}{9} n}}} \mathds{1}_{P_{i}}(p) \phi_{3}(W) \frac{3 p^2 \log p}{W N} \geq \frac{2}{9} \frac{f_{i}\left(b_{i}\right)+6 \kappa}{1+2 \kappa} \geq \kappa$$
and
\begin{align*}\sum_{i=1}^{9} \alpha_{i}' & \geq \frac{2}{9} \frac{1}{1+2 \kappa}\left(\sum_{i=1}^{9} f_{i}\left(b_{i}\right)+54 \kappa\right) \\ & \geq \frac{1}{1+2 \kappa}\left(\frac{134}{81}+12 \kappa\right) \\ & \geq \frac{134}{81}+7 \kappa.\end{align*}
From now on, we consider the sets $A_{i}$, $1 \leq i \leq 9$, as subsets of $\mathbb{Z}_{N}$. Since $A_{i} \subseteq[0, \frac{2 n}{9 W}]$ and let $N \geq\frac{n}{W}+13$, we have that $n' \in A_{1}+A_{2}+\cdots+A_{9} \in \mathbb{Z}_{N}$ is equivalent to $n' \in A_{1}+A_{2}+\cdots+A_{9} \in \mathbb{Z} $. Define $a_{i}(x)=\mathds{1}_{A_{i}}(x) \lambda_{i}(x)$.
For any function $f$ defined on $\mathbb{Z}_{N}$, define $\tilde{f}$ on $\mathbb{Z}_{N}$ by
$$\tilde{f}(r)=\sum_{x \in \mathbb{Z}_{N}} f(x) e\left(-\frac{x r}{N}\right),$$
and for functions $f, g$ on $\mathbb{Z}_{N}$, let
$$f \ast g(x)=\sum_{y \in \mathbb{Z}_{N}} f(y) g(x-y).$$
Next, we choose that $\epsilon=\frac{1}{\log \omega}$ and $\delta=\frac{1}{\log \log \omega}$. For each $i$, let $R_{i}=\{r \in \mathbb{Z}_{N}:|\tilde{a}_{i}(r)| \geq\delta\}$ and $B_{i}=\{x \in \mathbb{Z}_{N}:\parallel\frac{x r}{N}\parallel \leq\epsilon \ for\ all\ r \in R_{i}\}$, where $\parallel x\parallel=\min _{z \in \mathbb{Z}}|x-z|$. Moreover, let $\beta_{i}=\frac{\mathds{1}_{B_{i}}}{|B_{i}|}$ and $a_{i}'=a_{i} \ast\beta_{i} \ast \beta_{i}$.

The quartic analogue of \cite[Lemma 4.3]{tan} follows by the same argument, using \cite[Lemma 5.1]{Chow} with $d=3$ and $t=4$, as justified by the discussion following \cite[(1.11)]{Chow}. Thus, for every $q>8$ and every function $g_i:\mathbb{Z}_N\to\mathbb{C}$ satisfying $|g_i|\leq\lambda_i$, we have
\begin{align}\label{xin}\sum_{r\in\mathbb{Z}_N}|\widetilde{g}_i(r)|^q \ll_q 1. \end{align}
In the following argument, we take $q=17/2$.
\begin{lemma}\label{lemma4.5}
For some constant $C_1$ we have
$$\left|\sum_{\substack{x_{1}, x_{2},\ldots , x_{9} \in \mathbb{Z}_{N} \\ x_{1}+x_{2}+\cdots+x_{9}=n'}} \prod_{i=1}^{9} a_{i}'\left(x_{i}\right)-\sum_{\substack{x_{1}, x_{2}, \ldots, x_{9} \in \mathbb{Z}_{N} \\ x_{1}+x_{2}+\cdots+x_{9}=n'}} \prod_{i=1}^{9} a_{i}\left(x_{i}\right)\right| \leq \frac{C_{1}}{N}\left(\epsilon^{2} \delta^{-\frac{17}{2}}+\delta^{\frac{1}{18}}\right).$$
\end{lemma}\begin{proof}
We can get the conclusion by replacing $\frac{1}{10}$ with $\frac{1}{18}$, replacing $\frac{9}{10}$ with $\frac{17}{18}$, and replacing $\frac{1}{5}$ with $\frac{1}{9}$ in \cite[(4.5)]{tan}.
\end{proof}
\begin{lemma}\label{lemma4.6}
Fix a constant $C>1$, there exists $C_{2}$ depending only on $C$ such that whenever $\epsilon^{|R_{i}|} \geq C_{2} \omega^{-\frac{1}{4}}$, we have
\begin{align}\label{4.6}\left|a_{i}'(x)\right| \leq \frac{1+2 C^{-1}}{N}. \end{align}
for each $x \in \mathbb{Z}_{N}$.
\end{lemma}\begin{proof}
See \cite[Lemma 4.6]{tan}. \end{proof}

\begin{lemma}\label{lemma4.7}
For non-empty subsets $X_{1}, X_{2}, ..., X_{k}$ of $\mathbb{Z}_{N}$, define
$$v_{X_{1},X_{2},... ,X_{k}}(n)=\# \{ (x_{1},x_{2},... ,x_{k}):x_{i}\in X_{i},\, n=x_{1}+x_{2}+\cdots +x_{k}\}.$$
Suppose that $k \geq2$, $0<\theta_{1}, \theta_{2}, ..., \theta_{k} \leq1$ and $\theta_{1}+\cdots+\theta_{k}>1$. Let
$$\theta=\min \left\{\theta_{1}, ..., \theta_{k}, \frac{\theta_{1}+\cdots+\theta_{k}-1}{3 k-5}\right\}.$$
Let $N$ be a prime larger than $2 \theta^{-2}$, and let $X_{1}, ..., X_{k}$ be subsets of $\mathbb{Z}_{N}$ with $|X_{i}| \geq\theta_{i} N$. Then, for any $n \in \mathbb{Z}_{N}$, we have $v_{X_{1}, X_{2}, ..., X_{k}}(n) \geq\theta^{2 k-3} N^{k-1}$.
\end{lemma}\begin{proof}
See \cite[Lemma 3.3]{lipan2010}.
\end{proof}
\begin{lemma}\label{lemma4.8}
We have
\begin{align}\label{4.7}
\sum_{\substack{x_{1}, x_{2}, \ldots, x_{9} \in \mathbb{Z}_{N} \\ x_{1}+x_{2}+\cdots+x_{9}=n'}} \prod_{i=1}^{9} a_{i}'\left(x_{i}\right) \geq \frac{\kappa^{33}}{2^{15} N}.\end{align}
\end{lemma}\begin{proof}
Let $A_{i}'=\{x \in \mathbb{Z}_{N}: a_{i}'(x) \geq\frac{\alpha_{i}' \kappa}{N}\}$. Here, $R_{i}$ and $\delta$ are defined as above and for any $q>8$, applying (\ref{xin}), we get
$$|R_{i}|\delta^q\leq \sum_{r \in R_{i}}|\tilde{a}_{i}(r)|^q\leq\sum_r|\tilde{a}_{i}(r)|^q\ll_q1.$$
Let $q=\frac{17}{2}$, we have $|R_i|\leq C'_1\delta^{-\frac{17}{2}}$ for some absolute constant $C'_1$. For sufficiently large $n$, the condition of Lemma \ref{lemma4.6} is satisfied. By Lemma \ref{lemma4.6} with $C=\frac{2}{\kappa}$, we have
\begin{align*}\alpha_{i}' =\sum_{x \in \mathbb{Z}_{N}} a_{i}(x)=\sum_{x \in \mathbb{Z}_{N}} a_{i}'(x)\leq \frac{1+\kappa}{N}\left|A_{i}'\right|+\frac{\alpha_{i}' \kappa}{N}\left(N-\left|A_{i}'\right|\right). \end{align*}
Thus
$$\left|A_{i}'\right| \geq \frac{\alpha_{i}'(1-\kappa)}{1+\kappa} N.$$
By noticing that $\frac{\alpha_{i}'(1-\kappa)}{1+\kappa} \geq\frac{\kappa}{2}$ and $\sum_{i=1}^{9} \frac{a_{i}'(1-\kappa)}{1+\kappa} \geq\frac{134}{81}+3 \kappa$ and applying Lemma \ref{lemma4.7}, we get
$$v_{A_{1}', A_{2}', \ldots, A_{9}'}\left(n'\right) \geq \frac{\kappa^{15}}{2^{15}} N^{8}.$$
It follows that
\begin{align*}\sum_{\substack{x_{1}, x_{2}, \ldots, x_{9} \in \mathbb{Z}_{N} \\ x_{1}+x_{2}+\cdots+x_{9}=n'}} \prod_{i=1}^{9} a_{i}'\left(x_{i}\right) & \geq \sum_{\substack{x_{i} \in A_{i}', i=1, 2, \ldots, 9 \\ x_{1}+x_{2}+\cdots+x_{9}=n'}} \prod_{i=1}^{9} a_{i}'\left(x_{i}\right) \\ & \geq \frac{\kappa^{15} N^{8}}{2^{15}} \frac{\alpha_{1}' \alpha_{2}' \cdots\alpha_{9}' \kappa^{9}}{N^{9}} \\ & \geq \frac{\kappa^{33}}{2^{15}} N^{-1},\end{align*}
which completes the proof.
\end{proof}
Now combining Lemma \ref{lemma4.5} and Lemma \ref{lemma4.8}, we get
$$N \sum_{\substack{x_{1}, x_{2}, \ldots, x_{9} \in \mathbb{Z}_{N}\\ x_{1}+x_{2}+\cdots+x_{9}=n'}} \prod_{i=1}^{9} a_{i}\left(x_{i}\right)+C_{1}\left(\epsilon^{2} \delta^{-\frac{17}{2}}+\delta^{\frac{1}{18}}\right) \geq \frac{\kappa^{33}}{2^{15}}.$$
By the choice of $\epsilon=\frac{1}{\log \omega}$ and $\delta=\frac{1}{\log \log \omega}$, then for sufficiently large $n$, we have
$$\sum_{\substack{x_{1}, x_{2}, \ldots, x_{9} \in \mathbb{Z}_{N}\\ x_{1}+x_{2}+\cdots+x_{9}=n'}} \prod_{i=1}^{9} a_{i}\left(x_{i}\right)>0,$$
and hence there exist $x_{i} \in A_{i}$ such that $\sum_{i=1}^{9} x_{i}=n'$, and the proof of Theorem \ref{THM2} is complete.

\subsection*{Acknowledgements}
This work is supported by the National Natural Science Foundation of China (Grant Nos. 12571002, 12501007 and 12526407), Natural Science Foundation of Henan Province (Grant No. 252300421777), Young Elite Scientists Sponsorship Program by Henan Association for Science and Technology (Grant No. 2026HYTP001), High-level Talent Research Start-up Project Funding of Henan Academy of Sciences (Grant Nos. 241819033, 252019083 and 242019104) and High-level Achievement Award and Cultivation Project Funding of Henan Academy of Sciences (Grant Nos. 20262319007, 20262319001 and 20262319003).

\bibliographystyle{amsplain}

\begin{thebibliography}{10}
\bibitem{Chow}S. Chow, Roth--Waring--Goldbach, Int. Math. Res. Not. IMRN (2018), no. 8, 2341--2374.
\bibitem {Gao}M. Gao, A density version of Waring-Goldbach problem, Int. J. Number Theory 21 (2025), no.6, 1417-1436.
\bibitem {green}B. Green, Roth’s theorem in the primes, Ann. of Math. (2) 161 (2005), no. 3, 1609-1636.
\bibitem {Hua}L.-K. Hua, Some results in the additive prime-number theory, Quart. J. Math. Oxford Ser. (2) 9 (1938), no. 1, 68-80.
\bibitem {Hua1}L.-K. Hua, Additive Theory of Prime Numbers, Translations of Mathematical Monographs, vol. 13, American Mathematical Society, Providence, R.I., 1965.
\bibitem {lipan2010}H. Li and H. Pan, A density version of Vinogradov’s three primes theorem, Forum Math. 22 (2010), no. 4, 699-714.
\bibitem {shao2014}X. Shao, A density version of the Vinogradov three primes theorem, Duke Math. J. 163 (2014), no. 3, 489–512.
\bibitem {tan}M. Tan, A density version of Hua’s theorem, Forum Math. 38 (2026), no. 2, 307-319.
\bibitem {tao}T. Tao, V. Vu, Additive Combinatorics, Cambridge Studies in Advanced Mathematics 105, Cambridge University Press, Cambridge, 2006.
\bibitem {vinogradov1937}I. M. Vinogradov, The representation of an odd number as a sum of three primes, Dokl. Akad. Nauk. SSSR. 16 (1937), 139–142.
\end{thebibliography}

\end{document}